\documentclass[11pt]{article}
\usepackage{amssymb}
\usepackage{latexsym}
\usepackage{graphicx}
\usepackage{amsmath}
\usepackage{hyperref}
\newenvironment{proof}{{\em Proof.} }{\hfill$\Box$\vspace{0.1in}}

\newtheorem{lemma}{Lemma}[section]
\newtheorem{theorem}{Theorem}[section]

\numberwithin{equation}{section}
\newcommand{\ep}{{\varepsilon}}
\newcommand{\ina}{{\langle i\nabla\rangle}}

\newcommand{\p}{{\partial}}

\title{Scattering problem for Klein-Gordon equation with cubic convolution nonlinearity
\thanks{Supported by NSFC (10931007) and Zhejiang Provincial
NSF of China (Y6090158)}}
\author{ Ruying  Xue\thanks{E-mail address: ryxue@zju.edu.cn}\\ {\small Department of  Mathematics, Zhejiang University,}\\
{\small Hangzhou 310027, Zhejiang, P.  R.  China }}
\date{}
 
\begin{document}
 \maketitle
\pagestyle{myheadings} \markboth{\hfill\small R.
Xue/Scattering Problem for Klein-Gordon equation \hfill} {\hfill\small R.
Xue/Scattering Problem for Klein-Gordon equation\hfill}
\begin{abstract}The scattering problem for the Klein-Gordon equation
with cubic convolution nonlinearity is considered. Based on the Strichartz estimates
for the inhomogeneous Klein-Gordon equation, we  prove the existence of the scattering operator, which improves the
known results in some sense.


{\bf {Subject class:}} 35P25, 81Q05, 35B05
\end{abstract}

\section{Introduction}

This paper is concerned with the scattering problem for the
nonlinear Klein-Gordon equation of the form
\begin{equation}\left\{\begin{array}{l}\partial_t^2u-\Delta u+u=F_\gamma (u)\qquad (t,x)\in {R}\times  {R}^n\\
u|_{t=0}=f(x),\, \partial_t u|_{t=0}=g(x)\end{array}\right.\label{1.1}\end{equation}
 where u is a real-valued or a complex-valued unknown function
of $(t, x)\in R\times R^n$. The nonlinearity
is a cubic convolution term  $F_\gamma (u)=-(V_\gamma (x)\ast |u|^2)u$ with
$|V_\gamma(x)|\le C|x|^{-\gamma}$. Here, $0 <\gamma< n$ and $\ast$ denotes the convolution in the space variables. The term $F_\gamma (u)$ is an approximative expression of the nonlocal interaction of specific elementary
particles. The equation (\ref{1.1}) was studied by Menzala and Strauss  in \cite{[1]}.

In order to   define the scattering operator for (\ref{1.1}), we first give some Banach spaces.
The usual Lebesgue space is given by $L^p=\{\phi\in S^\prime: \|\phi\|_{L^p}<+\infty\}$, where the norm $\|\phi\|_{L^p}=\{\int_{R^n}|\phi(x)|dx\}^{1/p}$ if $1\le p<+\infty$ and $\|\phi\|_{L^\infty}=\sup _{x\in R^n}|\phi(x)|$ if $p=+\infty$. The weighted Sobolev space $H^{\beta,k}_p$ is defined by
$$H^{\beta,k}_p=\{\phi\in S^\prime: \|\phi\|_{H^{\beta,k}_p}=\|\langle x\rangle^k\ina^\beta\phi\|_{L^p}<+\infty\},$$
with $\langle x\rangle=\sqrt{1+x^2}$ and $\ina=\sqrt{1-\bigtriangleup}$. We also write $H^{\beta,k}=H^{\beta,k}_2$ and $H^\beta=H^{\beta,0}_2$ if it does not cause a confusion. A Hilbert space $X^{\beta,k}$ is denoted by $H^{\beta,k}\bigoplus H^{\beta-1,k}$. Let $X_\rho^{\beta,k}$ be
a ball of a radius $\rho>0$ with a center in the origin in the space $X^{\beta,k}$. The scattering operator of (\ref{1.1}) is defined as the mapping $S: X_\rho^{\beta,k} \ni
(f_-, g_-)\to  (f_+, g_+)\in X^{\beta,0}$ if the following condition holds:\newline
{\it For   $(f_-, g_-)\in  X_\rho^{\beta,k}$, there uniquely exists  a time-global solution $u\in C(R; H^\beta)$
of (\ref{1.1}), and data $(f_+, g_+)\in X^{\beta, 0}$ such that $u(t)$ approaches $u_\pm(t)$ in
$H^\beta$ as $t\to\pm\infty$, where $u_\pm(t)$ are solutions of linear Klein-Gordon equations whose initial data are $(f_\pm, g_\pm)$, respectively.}

We say that  $(S,X^{\beta,k})$  is well-defined if we can define the scattering operator $S:   X_\rho^{\beta,k}\to   X^{\beta,0}$  for some $\rho>0$. In \cite{[2]},  Mochizuki prove that if $n\ge 3$, $\beta\ge 1$, $\gamma< n$ and $2\le\gamma\le 2\beta+2$,
then $(S,X^{\beta,0})$ is well-defined.  Hidano \cite{[3]} see that if $ n\ge 2$, $\beta\ge 1$, $4/3<\gamma<2$ and $k>1/3$, then $(S,X^{\beta,k})$ is well-defined.  By using the
Strichartz estimate for pre-admissible pair and the complex interpolation method
for the weighted Sobolev space,  Hidano \cite{[4]} shows  that $(S,X^{\beta, k})$ is well-defined if $n\ge 2$, $\beta\ge 1$, $4/3<\gamma<2$ and $k>(2-\gamma)/2$. Our aim of this article is to  show that $(S, X^{\beta,1})$ is well-defined if $n\ge 2$,
\begin{equation}1<\gamma<\min\{\frac {2(n+1)}{n+2},\, \frac{3n-2}{n+2}\}, \frac{(n+2)(\gamma+1)}{4n}+\frac 12<\beta< \frac{(n+2)(\gamma+1)}{2n}.\label{1.1a}\end{equation}
More precisely, we prove the following theorem.
\begin{theorem}\label{T1}Let the function $V_\gamma(x)$ satisfy
 $$|V_\gamma(x)|\le C|x|^{-\gamma}, \, |\nabla V_\gamma(x)|\le C|x|^{-(1+\gamma)}.$$
 Assume that $n\ge 2$, $\gamma$ and $\beta$ satisfy (\ref{1.1a}). Then  there exists a positive
number $\delta_0>0$ satisfying the following properties:

(1). For $(f,g)\in X^{\beta,1}$ with $\|(f,g)\|_{X^{\beta,1}}\le \delta_0$, there uniquely exist finial states
$(f_\pm, g_\pm)\in X^{\beta,0}$ and  a solution $u(t)\in C(R; H^\beta)$ of (\ref{1.1}) such that $u(t)$ approaches $u_\pm(t)$ in $X^{\beta,0}$ as $t\to\pm\infty$, where $u_\pm(t)$ are solutions of the linear Klein-Gordon equation with initial data $(f_\pm, g_\pm)$, respectively. Moreover, as $\pm t$ large enough we have
$$\|(u(t),\partial_t u(t))-(u_\pm (t),\partial_t u_\pm (t))\|_{X^{\beta,0}}\le C\langle t\rangle^{-\delta}$$
with $\delta=\frac{2n\beta}{n+2}-2>0$.

(2).  For $(f_-, g_-)\in X^{\beta,1}$ with $\|(f_-, g_-)\|_{X^{\beta,1}}\le \delta_0$, there uniquely exists a finial state $(f_+, g_+)\in X^{\beta,0}$ and  a solution $u(t)\in C(R; H^\beta)$ of (\ref{1.1}) such that $u(t)$ approaches $u_\pm(t)$ in $X^{\beta,0}$ as $t\to\pm\infty$, where $u_\pm(t)$ are solutions of the linear Klein-Gordon equation with initial data $(f_\pm, g_\pm)$, respectively. Moreover, as $\pm t$ large enough we have
$$\|(u(t),\partial_t u(t))-(u_\pm (t),\partial_t u_\pm (t))\|_{X^{\beta,0}}\le C\langle t\rangle^{-\delta}$$
with $\delta=\frac{2n\beta}{n+2}-2>0$.\end{theorem}

In this article we denote by  $J_\ep=\ina x+i\ep t \nabla$ ,  $L_\ep=i\p_t-\ep\ina$ and $P=t\nabla+x\partial_t$ with $\ep\in\{+, -\}$. For a given Banach space with norm $\|\cdot\|$ and a vector $v=(v^+,v^-)$,  denote by
$$\|v\|=\|v^+\|+\|v^-\|, \, \|Pv\|=\|Pv^+\|+\|Pv^-\|, $$
$$\|Jv\|=\|J_+v^+\|+\|J_-v^-\|, \, \|Lv\|=\|L_+v^+\|+\|L_-v^-\|.$$  We also denote by the space-time norm
$$\|\phi\|_{L_t^r(I,L_x^q)}=\|\|\phi(t)\|_{L_x^q(R^n)}\|_{L_t^r(I)},$$
where I is a bounded or unbounded time interval, and denote different positive constants by the same letter C.

The rest of the article is organized as follows. In Section 2 we give some
preliminary calculations. Then Section 3 is devoted to the proof of Theorem \ref{T1}.

\section{Preliminaries}

 In this section, we prove some lemmas for our main results. Let  $w^{\ep}=i\partial_t\ina^{-1}u-\varepsilon u$  with $\ep\in \{+, -\}$. Then the Klein-Gordon equation (\ref{1.1}) can be be rewritten as a system of equations
\begin{equation}\left\{\begin{array}{l}L_\ep w^{\ep}=\ina^{-1}F_\gamma (u)\\
w^\ep|_{t=0}=w_0^\ep\end{array}\right.\label{1.3}\end{equation}
where $L_\ep=i\p_t-\ep\ina$, $w_0^\ep=i\ina^{-1}g+\ep f$. By the fact that
$$u=\frac 12(w^+-w^-), \partial_tu=-\frac i2\ina (w^++w^-),$$
we can rewrite the term $F_\gamma (u)$ as
$$F_\gamma (u)=\sum_{\ep_1,\ep_2,\ep_3\in \{+,-\}}C_{\ep_1\ep_2\ep_3}(V_\gamma\ast\overline{w^{\ep_1}}w^{\ep_2})w^{\ep_3}$$
with some constants $C_{\ep_1\ep_2\ep_3}$. Denote   $U_\ep (t)\varphi=e^{-\ep i\ina t}\varphi$ and for given $T\in R$,
$$\Psi_\ep [g]=\int_T^t U_\ep (t-\tau)\ina^{-1}g(\tau)d\tau,$$

\begin{lemma}\label{L1.1} Let $2\le q<\frac{2n}{n-2}$, $\frac 2r=\frac n2(1-\frac 2q)$.  Then for any time interval I and for any given $T\in I$
the following estimates are true:
$$\|\Psi_\ep [g]\|_{L_t^r(I,L^q)}\le \|g\|_{L_t^{r\prime}(I,H_{q\prime}^{2\mu-1})},$$
$$\|\Psi_\ep [g]\|_{L_t^\infty(I,L^2)}\le \|g\|_{L_t^{r\prime}(I,H_{q\prime}^{\mu-1})},$$and
$$\|U_\ep(t)\varphi\|_{L_t^r(I,L^q)}\le \|\varphi\|_{H^\mu},$$
where $r^\prime=\frac r{r-1}$, $q^\prime=\frac q{q-1}$ and $\mu=\frac 12(1+\frac n2)(1-\frac 2q)$.\end{lemma}
The proof of Lemma \ref{L1.1} is based on  the duality argument   along with the
$L^p-L^q$ time decay estimates. The similar result be found  in \cite{[5]}.

\begin{lemma}\label{L1.2}Assume $2\le p<\frac{2n}{n-2}$ for $n\ge3$ ($2\le p<+\infty$ for   $n=2$), denote by $\alpha=(1+\frac n2)(1-\frac 2p)$. The estimate is valid
	$$\|\phi\|_{L^p}\le C\langle t\rangle^{-\frac n2(1-\frac 2p)}(\|\phi\|_{H^\alpha}+\|J_\ep\phi\|_{H^{\alpha-1}}),$$
for  all $t\in R$, provided that the right-hand side is finite.\end{lemma}
This lemma comes from Lemma 2.1 in \cite{[5]} and the fact that $\|\phi\|_{L^p}\le C\|\phi\|_{H^\alpha}$ when  $p\ge 2$.

\begin{lemma}\label{L1.3}Assume $|V_\gamma (x)|\le |x|^{-\gamma}$ with $0<\gamma<n$,  $\ep_1, \ep_2, \ep_3\in\{+, -\}$.

(1). For $1<r<+\infty$, $1<p_1, p_2<+\infty$ and $p_3>r$   satisfying $1+\frac 1r=\frac\gamma n+\frac 1{p_1}+\frac 1{p_2}+\frac 1{p_3}$, we have
 $$\|(V_\gamma\ast\overline{w^{\ep_1}}w^{\ep_2})w^{\ep_3}\|_{L^r}\le \|w^{\ep_1}\|_{L^{p_1}} \|w^{\ep_2}\|_{L^{p_2}} \|w^{\ep_3}\|_{L^{p_3}}.$$

 (2). For $\rho>0$, $1<r<+\infty$, $1<p_{jk}<+\infty$ for $j, k\in\{1,2\}$ and $p_{13}, p_{23}>r$ satisfying $1+\frac 1r=\frac\gamma n+\frac 1{p_{j1}}+\frac 1{p_{j2}}+\frac 1{p_{j3}}$, we have
 \begin{eqnarray*}\|(V_\gamma\ast\overline{w^{\ep_1}}w^{\ep_2})w^{\ep_3}\|_{H^\rho_r}&\le& \|w^{\ep_1}\|_{H^\rho_{p_{11}}} \|w^{\ep_2}\|_{L^{p_{12}} } \|w^{\ep_3}\|_{L^{p_{13}}}+\|w^{\ep_1}\|_{L^{p_{12}}} \|w^{\ep_2}\|_{H^\rho_{p_{11}} } \|w^{\ep_3}\|_{L^{p_{13}}}\\
 & &+\|w^{\ep_1}\|_{L^{p_{21}}} \|w^{\ep_2}\|_{L^{p_{22}} } \|w^{\ep_3}\|_{H^\rho_{p_{23}}}\end{eqnarray*}
\end{lemma}
\begin{proof}To prove (1), put $\frac 1{p_4}=\frac 1r-\frac 1{p_3}$. By the H\"{o}lder inequality and the Hardy-Littlewood-Sobolev inequality, we have
\begin{eqnarray*} \|(V_\gamma\ast\overline{w^{\ep_1}}w^{\ep_2})w^{\ep_3}\|_{L^r}&\le&\|V_\gamma\ast\overline{w^{\ep_1}}w^{\ep_2}\|_{L^{p_4}}
 \|w^{\ep_3}\|_{L^{p_3}}\\
&\le&\|w^{\ep_1}\|_{L^{p_1}} \|w^{\ep_2}\|_{L^{p_2}} \|w^{\ep_3}\|_{L^{p_3}}\end{eqnarray*}
since  $1+\frac 1{p_4}=\frac\gamma n+\frac 1{p_1}+\frac 1{p_2}$.

 To prove (2), we set $\frac 1r=\frac 1{p_{14}}+\frac 1{p_{13}}$ and $\frac 1r=\frac 1{p_{24}}+\frac 1{p_{23}}$. For $\rho>0$,  the  generalized H\"{o}lder inequality in \cite{[6]} implies
$$\|(V_\gamma\ast\overline{w^{\ep_1}}w^{\ep_2})w^{\ep_3}\|_{H^\rho_r}\le  \|V_\gamma\ast\overline{w^{\ep_1}}w^{\ep_2}\|_{H^\rho_{p_{14}} } \|w^{\ep_3}\|_{L^{p_{13}}}+
 \|V\ast\overline{w^{\ep_1}}w^{\ep_2}\|_{L^{p_{24}}}  \|w^{\ep_3}\|_{H^\rho_{p_{23}} }$$
  By the generalized H\"{o}lder inequality and the Hardy-Littlewood-Sobolev inequality, we have
\begin{eqnarray*}\|V_\gamma\ast\overline{w^{\ep_1}}w^{\ep_2}\|_{H^\rho_{p_{14}} } &\le& \|V_\gamma\ast\ina^\rho(\overline{w^{\ep_1}}w^{\ep_2})\|_{L^{p_{14}}}\le\|\ina^\rho(\overline{w^{\ep_1}}w^{\ep_2})\|_{L^{p_{15}}} \\
&\le&  \|w^{\ep_1}\|_{H^\rho_{p_{11}}} \|w^{\ep_2}\|_{L^{p_{12}} } + \|w^{\ep_2}\|_{H^\rho_{p_{11}}} \|w^{\ep_1}\|_{L^{p_{12}} }\end{eqnarray*}
since $1+\frac 1{p_{14}}=\frac\gamma n+\frac 1{p_{15}}$ and $\frac 1{p_{15}}=\frac 1{p_{11}}+\frac 1{p_{12}}$.
Similarly we have
$$ \|V_\gamma\ast\overline{w^{\ep_1}}w^{\ep_2}\|_{L^{p_{24}}}\le \|w^{\ep_1}\|_{L^{p_{21}}} \|w^{\ep_2}\|_{L^{p_{22}} }.$$
\end{proof}

\section{Proof of Theorem \ref{T1}}

For $1<\gamma<\min\{\frac {2(n+1)}{n+2}, \frac{3n-2}{n+2}\}$, we choose
$$ \frac{(n+2)(\gamma+1)}{4n}+\frac 12<\beta< \frac{(n+2)(\gamma+1)}{2n}, q=\left(\frac{2\beta}{n+2}+\frac 12-\frac{\gamma+1}n\right)^{-1},$$
They satisfy
$$1\le\beta\le 2, 2<q<\frac{2n}{n+2(1-\gamma)}, 1<\gamma<\frac{3n\beta}{n+2}.$$
Let $\mu=\frac 12(1+\frac n2)(1-\frac 2q)$, we also have
$$\mu+\beta-2\le 0, \mu\le\beta-1, \mbox{and} \quad 0<\mu\le\frac 12.$$
Let $r, p$ and $s$ be chosen as
$$\frac 2r=\frac n2(1-\frac 2q),  \frac 2p+\frac\gamma n=2-\frac 2q, \frac 2s=1-\frac 2r.$$

{\it The proof of Theorem \ref{T1}(1).} Introduce the function space
$$X=\{v=(v^+, v^-)\in C(R;(L^2(R^n))^2); \quad \|v\|_X<+\infty\}$$
with the norm
\begin{eqnarray*}\|v\|_X&=&\|v\|_{L_t^\infty (R,H^\beta)}+\|v\|_{L_t^r(R,H^{\beta-\mu}_q)}+\|\partial_t v\|_{L_t^\infty (R,H^{\beta-1})}
+\|\partial_tv\|_{L_t^r(R,L^q)}\\
& &+\|Pv\|_{L_t^\infty (R,H^{\beta-1})}
+\|Pv\|_{L_t^r(R,L^q)}+\|Jv\|_{L_t^\infty (R,H^{\beta-1})}.\end{eqnarray*}
Denote by $X_\rho$ a ball of a radius $\rho>0$ with a center in the origin in the space $X$. Let us
consider the linearized version of (\ref{1.3})
\begin{equation}\left\{\begin{array}{l}L_\ep w^{\ep}=\ina^{-1}F_\gamma (v)\\
w^\ep|_{t=0}=w_0^\ep\end{array}\right.\label{3.3}\end{equation}
with a given vecter $v=(v^+,v^-)\in X_\rho$, where
$$F_\gamma (v)=\sum_{\ep_1,\ep_2,\ep_3\in \{+,-\}}C_{\ep_1\ep_2\ep_3}(V_\gamma\ast \overline{{v^{\ep_1}}}v^{\ep_2})v^{\ep_3}$$
with some given constants $C_{\ep_1\ep_2\ep_3}$.  The integration of the  linearized  Cauchy problem (\ref{3.3}) with respect to
time yields
\begin{equation}w^\ep=U_\ep (t)w_0^\ep+\sum_{\ep_1,\ep_2,\ep_3\in \{+,-\}}C_{\ep_1\ep_2\ep_3}\Psi_\ep((V_\gamma\ast \overline{{v^{\ep_1}}}v^{\ep_2})v^{\ep_3}).\label{3.4}\end{equation}
Taking the $L_t^\infty(R;H^\beta)$-norm of (\ref{3.4}), applying the H\"{o}lder inequality,  Lemma \ref{L1.1} and  Lemma \ref{L1.3}, we find
\begin{eqnarray}\|w^\ep\|_{L_t^\infty(R;H^\beta)}&\le&\|U_\ep (t)w_0^\ep\|_{L_t^\infty(R;H^\beta)}+C\sum_{\ep_1,\ep_2,\ep_3\in \{+,-\}} \|\Psi_\ep((V_\gamma\ast \overline{{v^{\ep_1}}}v^{\ep_2})v^{\ep_3})\|_{L_t^\infty(R;H^\beta)}\nonumber\\
&\le&\|w_0^\ep\|_{H^\beta}+C\sum_{\ep_1,\ep_2,\ep_3\in \{+,-\}} \|(V_\gamma\ast \overline{{v^{\ep_1}}}v^{\ep_2})v^{\ep_3}\|_{L_t^{r^\prime}(R;H^{\beta+\mu-1}_{q^\prime})}\nonumber\\
&\le&\|w_0\|_{H^\beta}+C\sum_{\ep_1,\ep_2,\ep_3\in \{+,-\}}\left \|\|{v^{\ep_1}}\|_{H^{\beta+\mu-1}_q} \|v^{\ep_2}\|_{L^p}\|v^{\ep_3}\|_{L^p}\right\|_{L_t^{r^\prime}(R)}\nonumber\\
&\le&\|w_0\|_{H^\beta}+ C\|{v}\|_{L_t^{r}(R; H^{\beta-\mu}_q)} \|v\|^2_{L^s_t(R; L^p)}\nonumber\\
&\le&\|w_0\|_{H^\beta}+C\rho  \|v\|^2_{L^s_t(R; L^p)}\label{3.5}\end{eqnarray}
since $p>2>q^\prime$, $q>2>q^\prime$, $\mu\le\frac 12$ and $2-\frac 2q=\frac\gamma n+\frac 2p$.  Similarly,  taking the $L_t^{r}(R; H^{\beta-\mu}_q)$  we obtain
\begin{eqnarray}\|w^\ep\|_{L_t^{r}(R; H^{\beta-\mu}_q)}&\le&\|U_\ep (t)w_0^\ep\|_{L_t^{r}(R; H^{\beta-\mu}_q)}+C\sum_{\ep_1,\ep_2,\ep_3\in \{+,-\}} \|\Psi_\ep((V_\gamma\ast \overline{{v^{\ep_1}}}v^{\ep_2})v^{\ep_3})\|_{L_t^{r}(R; H^{\beta-\mu}_q)}\nonumber\\
&\le&\|w_0^\ep\|_{H^\beta}+C\sum_{\ep_1,\ep_2,\ep_3\in \{+,-\}} \|(V_\gamma\ast\ \overline{{v^{\ep_1}}} v^{\ep_2})v^{\ep_3}\|_{L_t^{r^\prime}(R;H^{\beta+\mu-1}_{q^\prime})}\nonumber\\
&\le &\|w_0^\ep\|_{H^\beta}+C\|{v}\|_{L_t^{r}(R; H^{\beta+\mu-1}_q)} \|v\|^2_{L^s_t(R; L^p)}\nonumber\\
&\le&\|w_0 \|_{H^\beta}+C\rho  \|v\|^2_{L^s_t(R; L^p)}\label{3.6}\end{eqnarray}
since   $ \mu \le\frac 12$,  $p>2>q^\prime$, $q>2>q^\prime$ and $2-\frac 2q=\frac\gamma n+\frac 2p$.
Applying the operator $\partial_t$ to  (\ref{3.3}) we deduce that $\partial_t w^\ep$ satisfies the following system
\begin{equation*}\left\{\begin{array}{l}L_\ep \partial_t w^\ep=\ina^{-1}\partial_tF_\gamma (v)\\
\partial_t w^\ep|_{t=0}=-i\ep\ina w_0^\ep-i\ina^{-1}F_\gamma(v)|_{t=0}\end{array}\right.\end{equation*}
 with
$$F_\gamma (v)=\sum_{\ep_1,\ep_2,\ep_3\in \{+,-\}}C_{\ep_1\ep_2\ep_3}(V_\gamma\ast\overline{v^{\ep_1}}v^{\ep_2})v^{\ep_3}.$$
Then by integrating with respect to time,
$$\partial_tw^\ep=U_\ep (t)(\partial_tw^\ep|_{t=0})+\Psi_\ep(\partial_tF_\gamma (v)).$$
Taking the $L_t^\infty(R;H^{\beta-1})$-norm and $L^r_t(R,L^q)$-norm,  applying the H\"{o}lder inequality  and Lemma \ref{L1.1}  we find that, since $\beta\ge 1$, $\mu\le\beta-1$ and $\mu+\beta-2\le 0$,
\begin{eqnarray*}& &\|\partial_t w^\ep\|_{L_t^{\infty}(R; H^{\beta-1})}+\|\partial_tw^\ep\|_{L_t^{r}(R;L^q)}\\
&\le&\|\partial_tw^\ep|_{t=0}\|_{H^{\beta-1}}+\|\partial_t F_\gamma (v)\|_{L^{r^\prime}_t(R, H^{\mu+\beta-2}_{q^\prime})}\\
&\le&\|\partial_tw^\ep|_{t=0}\|_{H^{\beta-1}}\\
&&\quad +C\sum_{\ep_1,\ep_2,\ep_3\in \{+,-\}} \|(V_\gamma\ast (\overline{\partial_tv^{\ep_1}}v^{\ep_2}+\overline{v^{\ep_1}}\partial_tv^{\ep_2}))v^{\ep_3}+(V_\gamma\ast \overline{v^{\ep_1}}v^{\ep_2})\partial_tv^{\ep_3}\|_{L_t^{r^\prime}(R;L^{q^\prime})}\\
&\le&\|\partial_tw^\ep|_{t=0}\|_{H^{\beta-1}}+C\|\partial_tv\|_{L^r_t(R,L^q)}\|v\|^2_{L^s_t(R,L^p)}\\
&\le&\|\partial_tw^\ep|_{t=0}\|_{H^{\beta-1}}+C\rho\|v\|^2_{L^s_t(R,L^p)}\end{eqnarray*}
On the other hand, we have
$$\|\partial_tw^\ep|_{t=0}\|_{H^{\beta-1}}\le \|w^\ep_0\|_{H^{\beta}}+\|F_\gamma (v)\|_{L^\infty_t(R,H^{\beta-2})},$$
 and for $p_1>2$ satisfying $\frac 32=\frac\gamma{n}+\frac 3{p_1}$,
\begin{eqnarray*}\|F_\gamma (v)\|_{L^\infty_t(R,H^{\beta-2})}&\le& C\sum_{\ep_1,\ep_2,\ep_3\in \{+,-\}} \|(V_\gamma\ast \overline{ v^{\ep_1}}
v^{\ep_2})v^{\ep_3})\|_{L_t^{\infty}(R;L^2)}\\
&\le& C\|v\|^3_{L_t^{\infty}(R;L^{p_1})}\le  C\|v\|^3_{L_t^{\infty}(R;H^\beta)}\le C\rho^3\end{eqnarray*}
since $\beta\le 2$,$ \gamma\le 3\beta$ and  $\|v\|_{L^{p_1}} \le  C\|v\|_{H^\beta}$. Then
\begin{equation}\|\partial_t w^\ep\|_{L_t^{\infty}(R; H^{\beta-1})}+\|\partial_tw^\ep\|_{L_t^{r}(R;L^q)}\le C\|w_0\|_{H^\beta}+C\rho^3+C\rho \|v \|^2_{L^s_t(R,L^p)}.\label{3.9}\end{equation}

Notice that $P=t\bigtriangledown+x\partial_t$, $J_\ep=\ina x+i\ep t \bigtriangledown$ and $L_\ep=i\partial_t-\ep\ina$. We get
$$J_\ep=i\ep P-\ep L_\ep, \, [L_\ep, P]=-i\ep\ina^{-1}\bigtriangledown L_\ep,$$
$$[x, \ina]=\ina^{-1}\bigtriangledown,\,  [P, \ina^{-1}]=\ina^{-3}\bigtriangledown\partial_t$$
and
$$P((V_\gamma\ast \overline{v^{\ep_1}}v^{\ep_2})v^{\ep_3})=(V_\gamma\ast \overline{{v^{\ep_1}}}v^{\ep_2})P(v^{\ep_3})+(t\nabla V_\gamma\ast \overline{{v^{\ep_1}}}v^{\ep_2})v^{\ep_3}.$$
Applying the operator $P$ to (\ref{3.3}) yields
\begin{equation*}\left\{\begin{array}{l}L_\ep Pw^\ep=i\ep \ina^{-2}\nabla F_\gamma(v)-\ina^{-1}PF_\gamma(v)-\ina^{-3}\nabla \partial_tF_\gamma(v)\\
Pw^\ep|_{t=0}=x\partial_tw^\ep|_{t=0}=x(-i\ep\ina w_0^\ep-i\ina^{-1}F_\gamma(v)|_{t=0})\end{array}\right.\end{equation*}
 with
$$PF_\gamma (v)=\sum_{\ep_1,\ep_2,\ep_3\in \{+,-\}}C_{\ep_1\ep_2\ep_3}(V_\gamma\ast\overline{v^{\ep_1}}v^{\ep_2})Pv^{\ep_3}+(t\nabla V_\gamma\ast\overline{v^{\ep_1}}v^{\ep_2})v^{\ep_3}.$$
Integrating with respect to time, we get
\begin{equation}Pw^\ep=U_\ep (t)(Pw^\ep|_{t=0})-\Psi_\ep(i\ep\ina^{-1}\nabla F_\gamma(v))+\Psi_\ep(PF_\gamma(v))+\Psi_\ep(\ina^{-2}\nabla \partial_tF_\gamma (v)).\label{3.10}\end{equation}
Taking the $L_t^\infty(R;H^{\beta-1})$-norm  and the $L^r_t(R,L^q)$-norm of (\ref{3.10}), applying the H\"{o}lder inequality  and  Lemma \ref{L1.1},   we find
\begin{eqnarray}& &\|Pw^\ep\|_{L_t^{\infty}(R; H^{\beta-1})}+\|Pw^\ep\|_{L^r_t(R,L^q)}\nonumber\\
&\le& \|Pw^\ep|_{t=0}\|_{H^{\beta-1}}+\|\ina^{-1}\nabla F_\gamma\|_{L^{r^\prime}_t(R,L^{q^\prime})}\nonumber\\
&&\qquad +\|P F_\gamma(v)\|_{L_t^{r^\prime}(R; L^{q^\prime})}+\|\ina^{-2}\nabla \partial_t F_\gamma (v)\|_{L_t^{r^\prime}(R; L^{q^\prime})}\nonumber\\
&\le&\|Pw^\ep|_{t=0}\|_{H^{\beta-1}}+\|F_\gamma(v)\|_{L^{r^\prime}_t(R,L^{q^\prime})}\nonumber\\
&&\qquad +\|P F_\gamma(v)\|_{L_t^{r^\prime}(R; L^{q^\prime})}+\|\partial_t F_\gamma (v)\|_{L_t^{r^\prime}(R; L^{q^\prime})}\label{3.11}\end{eqnarray}
since $\beta\ge 1$ and $\mu+\beta-2\le 0$ and $\mu\le\beta-1$.
As in the proof of (\ref{3.9}) we deduce
\begin{eqnarray}& &\|F_\gamma(v)\|_{L^{r^\prime}_t(R,L^{q^\prime})}+\|\partial_t F_\gamma (v)\|_{L_t^{r^\prime}(R; L^{q^\prime})}\nonumber\\
&\le& C\|v\|_{L^r_t(R,L^q)}\|v\|^2_{L^s_t(R,L^p)}+C\|\partial_t v \|_{L^r_t(R,L^q)}\|v \|^2_{L^s_t(R,L^p)}
\le C\rho \|v\|^2_{L^s_t(R,L^p)}.\label{3.12}\end{eqnarray}
Let $p_3>2$ and $s_3>2$ satisfy
$$\frac 32-\frac 1q=\frac{\gamma+1}n+\frac 2{p_3}, 1-\frac 1r=\frac 2{s_3}.$$
The H\"{o}lder inequality and Lemma \ref{L1.3} imply
\begin{eqnarray}&&\|P F_\gamma(v)\|_{L_t^{r^\prime}(R; L^{q^\prime})}\nonumber\\
&\le&C\sum_{\ep_1,\ep_2,\ep_3\in \{+,-\}}\left[ \|(V_\gamma\ast \overline{v^{\ep_1}}v^{\ep_2})Pv^{\ep_3})\|_{L_t^{r^\prime}(R; L^{q^\prime})}+\|(t \nabla V_\gamma\ast \overline{v^{\ep_1}}v^{\ep_2})v^{\ep_3})\|_{L_t^{r^\prime}(R;L^{q^\prime})}\right]\nonumber\\
&\le& C\|Pv\|_{L^r_t(R,L^q)}\|v\|^2_{L^s_t(R,L^p)}+C\|v\|_{L^\infty_t(R,L^2)}\|t^{1/2} v \|^2_{L^{s_3}_t(R,L^{p_3})\nonumber}\\
&\le& C\rho\|v\|^2_{L^s_t(R,L^p)}+C\rho\|t^{1/2} v \|^2_{L^{s_3}_t(R,L^{p_3})},\label{3.13}\end{eqnarray}
here we use the condition $\|\nabla V_\gamma\|\le C|x|^{-(\gamma+1)}$.
By Lemma \ref{L1.2} we have
\begin{eqnarray}\|v\|_{L^s_t(R,L^p)}&\le& C\|\langle t\rangle^{-\frac n2(1-\frac 1p)}\left(\|v\|_{H^\alpha}+\|Jv\|_{H^{\alpha-1}}\right )\|_{L^s_t(R)}\nonumber\\
&\le& C \left (\|v\|_{L^\infty_t(R,H^\beta)}+\|J v\|_{L^\infty_t(R,H^{\beta-1})}\right )\le C\rho\label{3.14}\end{eqnarray}
since $\alpha=(1+\frac n2)(1-\frac 2p)\le \beta$ and $\frac n2(1-\frac 2p)>\frac 1s$. Similarly,
\begin{eqnarray}\|t^{1/2}v\|_{L^{s_3}_t(R,L^{p_3})}&\le& C\|\langle t\rangle^{-\frac n2(1-\frac 1{p_3})+\frac 12}\left(\|v\|_{H^{\alpha_3}}+\|Jv\|_{H^{\alpha_3-1}}\right )\|_{L^{s_3}_t(R)}\nonumber\\
&\le& C \left (\|v\|_{L^\infty_t(R,H^\beta)}+\|J v\|_{L^\infty_t(R,H^{\beta-1})}\right )\le C\rho,\label{3.15}\end{eqnarray}
since $\alpha_3=(1+\frac n2)(1-\frac 2{p_3})\le \beta$ and $\frac n2(1-\frac 2{p_3})>\frac 1{s_3}$.
Then we obtain, from (\ref{3.11})-(\ref{3.15}),
\begin{eqnarray}&&\|Pw^\ep\|_{L_t^{\infty}(R; H^{\beta-1})}+\|Pw^\ep\|_{L^r_t(R,L^q)}\le \|Pw^\ep|_{t=0}\|_{H^{\beta-1}}+C\rho^3,\label{3.16}\\
&&\|w^\ep\|_{L_t^\infty(R;H^\beta)}+\|w^\ep\|_{L_t^{r}(R; H^{\beta-\mu}_q)}\le\|w_0^\ep\|_{H^\beta}+C\rho^3,\label{3.16A}\\
&&\|\partial_t w^\ep\|_{L_t^{\infty}(R; H^{\beta-1})}+\|\partial_tw^\ep\|_{L_t^{r}(R;L^q)}\le\|w_0^\ep\|_{H^\beta}+C\rho^3,\label{3.16B}\end{eqnarray}

To estimate the term $ \|Pw^\ep|_{t=0}\|_{H^{\beta-1}}$,  we  give some estimates. It follows from the Sobolev embedding theorem that
\begin{eqnarray}\|F_\gamma(v)\|_{L^\infty_t(R,L^2)}&\le&C\sum_{\ep_1,\ep_2,\ep_3\in\{+,-\}}\|(V_\gamma\ast\overline{v^{\ep_1}}v^{\ep_2})
v^{\ep_3}\|_{L^\infty_t(R,L^2)}\nonumber\\
&\le& \|v\|^3_{L^\infty_t(R,L^{p_5})}\le C\|v\|^3_{L^\infty_t(R,H^\beta)}\le C\rho^3,\label{3.17}\end{eqnarray}
where  $p_5=\frac{6n}{3n-2\gamma}$, which satisfies $p_5\le\frac{2n}{n-2\beta}$ because of $\gamma\le3\beta$. Using the relation $x=\ina^{-1}J_\ep-i\ep t\ina^{-1}\nabla$ we deduce
\begin{eqnarray}& &\|x F_\gamma(v)\|_{L^\infty_t(R,L^2)}\le C\sum_{\ep_1,\ep_2,\ep_3\in\{+,-\}}\|(V_\gamma\ast\overline{v^{\ep_1}}v^{\ep_2})(xv^{\ep_3})\|_{L^\infty_t(R,L^2)}\nonumber\\
&\le&C\sum_{\ep_1,\ep_2,\ep_3\in\{+,-\}}\left\|\|v^{\ep_1}\|_{L^{p_4}}\|\|v^{\ep_2}\|_{L^{p_4}}\left(\|\ina^{-1}J_\ep v^{\ep_3})\|_{L^{p_4}}+t\|\ina^{-1}\nabla v^{\ep_3})\|_{L^{p_4}}\right)\right\|_{L^\infty_t(R)}\nonumber\\
 &\le&C\|v\|^2_{L^\infty_t(R,L^{p_4})} \|\ina^{-1}J v \|_{L^\infty_t(R,L^{p_4})}+C\|t^{1/3}v\|^3_{L^\infty_t(R,L^{p_4})}\nonumber\\
&\le&C\|v\|^2_{L^\infty_t(R,H^\beta)}\|Jv\|_{L^\infty_t(R,H^{\beta-1})}+C\left(\|v\|_{L^\infty_t(R,H^\beta)}+ \|Jv \|_{L^\infty_t(R,H^{\beta-1})}\right)^3\nonumber\\
 &\le&C\left(\rho+ \|Jv \|_{L^\infty_t(R,H^{\beta-1})}\right)^3\le C\rho^3,\label{3.18}\end{eqnarray}
 where $p_4=\frac{6n}{3n-2\gamma}$, which satisfies
$$2<p_4\le\frac{2n}{n-2\beta},\,  \frac n2(1-\frac 2{p_4})\ge\frac 13,\,  (1+\frac n2)(1-\frac 2{p_4})\le\beta$$
because of $1<\gamma\le\frac{3n\beta}{n+2}$. Using the relation $[\ina^{\beta-1}, x]=-(\beta-1)\ina^{\beta-3}\nabla$ we deduce
\begin{eqnarray*}& & \|Pw^\ep|_{t=0}\|_{H^{\beta-1}}\le\|x\ina w_0^\ep\|_{H^{\beta-1}}+\|x\ina^{-1}F_\gamma(v)\|_{L^\infty_t(R,H^{\beta-1})}\nonumber\\
&\le&\|x\ina w_0^\ep\|_{H^{\beta-1}}+\|\ina^{-1}x F_\gamma(v)\|_{L^\infty_t(R,H^{\beta-1})}+\|\ina^{-3}\nabla F_\gamma(v)\|_{L^\infty_t(R,H^{\beta-1})}\nonumber\\
&\le&\|\ina^{\beta-1}x\ina w_0^\ep\|_{L^2}+C\| x F_\gamma(v) \|_{L^\infty_t(R, L^2)}+C\| F_\gamma(v)\|_{L^\infty_t(R, L^2)}\nonumber\\
&\le&\|\langle x\rangle\ina^\beta w_0^\ep\|_{L^2}+C\rho^3+C\left(\rho+ \|Jv \|_{L^\infty_t(R,H^{\beta-1})}\right)^3\nonumber\\
&\le&\|w_0\|_{H^{\beta,1}}+C\rho^3,\label{3.19}\end{eqnarray*}
which, combining with (\ref{3.16}), yields
\begin{equation}\|Pw^\ep\|_{L_t^{\infty}(R; H^{\beta-1})}+\|Pw^\ep\|_{L^r_t(R,L^q)}\le \|w_0\|_{H^{\beta,1}}+C\rho^3.\label{3.20}\end{equation}

Notice that
$$[L_\ep,x]=-\ep\ina^{-1}\nabla,\,  [x,\ina^{-1}]=-\ina^{-3}\nabla.$$
Then we deduce that $xw^\ep$ satisfies
$$L_\ep(xw^\ep)=-\ep\ina^{-1}\nabla w^\ep-\ina^{-1}(xF_\gamma(v))+\ina^{-1}\nabla F_\gamma(v).$$
Using $J_\ep=i\ep P-\ep L_\ep x$ and (\ref{3.16A}) yields
$$\|J_\ep w^\ep\|_{L^\infty_t(R,H^{\beta-1})}\le \|P w^\ep\|_{L^\infty_t(R,H^{\beta-1})}+\|L_\ep(xw^\ep) \|_{L^\infty_t(R,H^{\beta-1})},$$
with
\begin{eqnarray*}&&\|L_\ep(xw^\ep) \|_{L^\infty_t(R,H^{\beta-1})}\\
&\le &\|w^\ep \|_{L^\infty_t(R,H^{\beta-1})}+\|\ina^{-2}F_\gamma(v)\|_{L^\infty_t(R,H^{\beta-1})}+\|\ina^{-1}(x F_\gamma(v))\|_{L^\infty_t(R,H^{\beta-1})}\\
&\le &\|w^\ep \|_{L^\infty_t(R,H^{\beta })}+\|F_\gamma(v)\|_{L^\infty_t(R,L^2)}+\|x F_\gamma(v)\|_{L^\infty_t(R,L^2)}
\le C\|w_0\|_{H^\beta}+C\rho^3.\end{eqnarray*} Then we get
\begin{equation}\|J_\ep w^\ep\|_{L^\infty_t(R,H^{\beta-1})}\le  C\|w_0\|_{H^{\beta,1}}+C\rho^3.\label{3.21}\end{equation}
A combination of (\ref{3.16}) with (\ref{3.16A}), (\ref{3.16B}), (\ref{3.20}) and (\ref{3.21}) yields
\begin{equation}\|w\|_X\le C\|w_0\|_{H^{\beta,1}}+C\rho^3.\label{3.22}\end{equation}
Therefore the map $M: w=M(v)$ defined by the problem (\ref{3.3}), transforms a ball $X_\rho$ with a small radius $\rho=C\|w^0\|_{H^{\beta, 1}}$ into itself. Denote
$\tilde w=M(\tilde v)$, then in the same way as in the proof of (\ref{3.22}) we have
$$\|M(v)-M(\tilde v)\|_X\le C\rho^2\|v-\tilde v\|_X.$$
Thus $M$ is a contraction mapping in $X_\rho$ and so there exists a unique solution $w =M(w)$ of (\ref{3.3}) if the norm $	 \|w^0\|_{H^{\beta, 1}}$ is small enough.

To prove the asymptotic of the solution $w(t,x)$, we use the equation, for $|t|>|t^\prime|$,
$$U_\ep(-t)w^\ep(t)-U_\ep(-t^\prime)w^\ep(t^\prime)=\int_{t^\prime}^tU_\ep(-\tau)\ina^{-1}F_\gamma(w(\tau))d\tau.$$
Taking the $H^\beta$-norm  of this equation, using the similar proof of (\ref{3.5}) and (\ref{3.6}), we deduce
$$\|U_\ep(-t)w^\ep(t)-U_\ep(-t^\prime)w^\ep(t^\prime)\|_{H^\beta}\le C\rho^2\langle t^\prime\rangle^{-\delta}$$
with $\delta=\frac{2n\beta}{n+2}-2>0$, since we have $\|w\|_X\le \rho$ and
$$\|\langle t\rangle^{-\frac n2(1-\frac 1p)}\|^2_{L^s_t([t^\prime, t])}\le C\langle t^\prime\rangle^{-\delta}.$$
Then there uniquely exist finial states $w^\ep_\pm\in H^\beta$ satisfying, for $\pm t$ large enough,
$$\|w^\ep(t)-U_\ep(t)w^\ep_\pm\|_{H^\beta}\le C\rho^2\langle t \rangle^{-\delta}.$$
Set $u(t)=\frac 12(w^+(t)-w^-(t))$, $f_\pm(x)=\frac 12(w_\pm^+-w_\pm^-)$, $ g_\pm(x)=-\frac i2\ina (w^+_\pm+w^-_\pm)$ and
$u_\pm(t)=\frac 12(U_+(t)w^+_\pm-U_-(t)w^-_\pm)$. Then $u(t)$ and $u_\pm(t)$ satisfy Theorem \ref{T1}(1).

{\it Proof of  Theorem \ref{T1}(2).} For given $(f_-,g_-)\in X^{\beta, 1}$ and $v=\{v^+, v^-)\in X_\rho$, we consider the linearized version of the final state problem of (\ref{3.3})
$$\left\{\begin{array}{l}L_\ep w^\ep=-\ina^{-1}F_\gamma(v)\\
\|U_\ep(t)w^\ep-w^\ep_-(x)\|_{H^\beta}\to 0 \, \mbox{ as} \, t\to\infty\end{array}\right.$$
with $w^\ep_-(x)=i\ina^{-1}g_-(x) -\ep f_-(x)\in H^{\beta,1}$. The integration with respective to time yields
$$w^\ep(t)=U_\ep(t)w^\ep_-+\int_{-\infty}^tU_\ep(t-\tau)\ina^{-1}F_\gamma(v(\tau))d\tau.$$
In the same way as in the proof of  Theorem \ref{T1}(1), we find that, if $\|(f_-,g_-)\|_{X^{\beta, 1}}\le\rho$ small,  there uniquely exists a global solution $w^\ep(t)\in C(R, H^\beta)$ and a final state $w_+^\ep\in H^\beta$ such that, as $t\to +\infty$,
$$\|w^\ep(t)-U_\ep(t)w_+^\ep\|_{H^\beta}\le C\langle t \rangle^{-\delta}$$
with $\delta=\frac{2n\beta}{n+2}-2>0$.
Set $u(t)=\frac 12(w^+(t)-w^-(t))$, $f_+(x)=\frac 12(w_+^+-w_+^-)$, $ g_+(x)=-\frac i2\ina (w^+_++w^-_+)$ and
$u_+(t)=\frac 12(U_+(t)w^+_+-U_-(t)w_+^-)$. Then $u(t)$ and $u_+(t)$ satisfy Theorem \ref{T1}(2).


\begin{thebibliography}{99}
\bibitem{[1]} G. Menzala and W. Strauss, On a wave equation with a cubic convolution, J. Differential
Equations, 43 (1982), 93-105.
\bibitem{[2]}K. Mochizuki, On small data scattering with cubic convolution nonlinearity, J. Math. Soc.
Japan, 41 (1989), 143-160.
\bibitem{[3]} K. Hidano, Small data scattering and blow-up for a wave equation with a cubic convolution,
Funkcialaj Ekvacioj, 43 (2000), 559-588.
\bibitem{[4]} K. Hidano, Small data scattering  for the Klein-Gordon equation with a cubic convolution nonlinearity,
    Discrete Cont. Dyn. Sys., 15(2006), 973-981.
\bibitem{[5]}N. Hayashi, P. I. Naumkin, Scattering operator for nonlinear Klein-Gordon equations in higher space dimensions, J. Differential Equations 244(2008), 188-199.
\bibitem{[6]}G. Ponce, On the global well-posedness of the Benjamin-Ono equation, Differential and Integral
Equations, 4 (1991), 527-542.
\bibitem{[8]} S. Klainerman, Global existence of small amplitude solutions to nonlinear Klein-Gordon equations in four space-time dimensions, Comm. Pure Appl. Math. 38 (1985), 631-641.
\bibitem{[9]} B. Marshall, W. Strauss, S. Wainger, $L^p-L^q$ estimates for the Klein-Gordon equation, J. Math. Pures Appl. 59(1980), 417-440.


\end{thebibliography}
\end{document}